\documentclass[11pt]{amsart}

\usepackage[USenglish]{babel} 
\usepackage{graphicx}
\graphicspath{ {images/} }
\usepackage[utf8]{inputenc}
\usepackage{float}
\usepackage{amssymb,verbatim,amscd,amsfonts,amsmath,amsthm,enumerate}
\usepackage[all]{xy}
\usepackage{subfig}
\usepackage{tikz}
\usetikzlibrary{shapes,arrows,shadows}
\usetikzlibrary{decorations.markings}
\usepackage{color}
\usepackage[normalem]{ulem}

\usepackage{hyperref}
\usepackage{wrapfig}
\usetikzlibrary{arrows}
\usepackage{pinlabel}

\long\def\symbolfootnote[#1]#2{\begingroup%
\def\thefootnote{\fnsymbol{footnote}}\footnote[#1]{#2}\endgroup}

\newtheorem{thm}{Theorem}[section]

\newtheorem{prop}[thm]{Proposition}
\newtheorem{corollary}[thm]{Corollary}
\newtheorem{lemma}[thm]{Lemma}

\theoremstyle{remark}

\theoremstyle{definition}

\newtheorem{remark}[thm]{Remark}

\newtheorem{defi}[thm]{Definition}

\newtheorem*{namedtheorem}{\theoremname}
\newcommand{\theoremname}{testing}

\newcommand{\N}{\mathbb{N}}
\newcommand{\Z}{\mathbb{Z}}

\def\Aut{\operatorname{Aut}}
\def\Mod{\operatorname{Mod}}

\def\rank{\operatorname{rank}}
\def\diam{\operatorname{diam}}

\makeatletter
\newcommand{\oset}[2]{%
  {\mathop{#2}\limits^{\vbox to -.5\ex@{\kern-\tw@\ex@
   \hbox{\scriptsize #1}\vss}}}}
\makeatother

\begin{document}
\title[Multicurve graphs on infinite-type surfaces]{Two remarks about multicurve graphs on infinite-type surfaces}	

\author{Julio Aroca}
\address{Instituto de Ciencias Matemáticas, C/ Nicolás Cabrera, 13-15, E-28049 Madrid, Spain}
\email{julio.aroca@icmat.es}

\date{\today}

\begin{abstract}
After Fossas-Parlier \cite{FP}, we consider two   graphs $\mathcal{G}_{0}(S)$ and $\mathcal{G}_{\infty}(S)$, constructed from multicurves on connected, orientable surfaces of infinite-type. 

Our first result asserts that  $\mathcal{G}_{\infty}(S)$ has finite diameter, which extends a result of Fossas-Parlier \cite{FP}. 
Next, we prove that the group of (label-preserving) automorphisms of  $\mathcal{G}_{0}(S)$ is the extended mapping class group of $S$, which may be regarded as an infinite-type analog of a theorem of Margalit \cite{DM} about pants complexes. 
\end{abstract}
\maketitle
\section{Introduction}
There has been a recent surge of interest in finding combinatorial models for mapping class groups of infinite-type surfaces, see \cite{AFP,AF,BA,BDR,DF,FP,HMF,HF,R}.

In \cite{FP}, A. Fossas and H. Parlier defined a family of graphs $\mathcal{G}_{k}(S)$ with $k \in \mathbb{N} \cup \{0, \infty\}$, constructed from multicurves on connected, orientable surfaces of infinite-type equipped with an upper bounded hyperbolic metric. 

In this note we will concentrate on the special cases $\mathcal{G}_{0}(S)$ and $\mathcal{G}_{\infty}(S)$, which we now briefly define; see Section \ref{S3} for a precise definition.

The vertices of $\mathcal{G}_{\infty}(S)$ are multicurves $\mu \subset S$ such that the supremum of the lengths of the curves in $\mu$ and the supremum of the complexities of the components of $S \backslash \mu$ are both finite. In \cite[\S 5]{FP}, it is proved that $\diam(\mathcal{G}_{\infty}(S)) \leq 3$ when $S$ has exactly two ends, each of which is non-planar. Our first objective is to generalize this result to arbitrary infinite-type surfaces: 

\begin{thm}\label{MT}
Let $S = S_\mathcal{H}$ be a connected orientable infinite-type upper-bounded hyperbolic surface. Then:
 $$\diam(\mathcal{G}_{\infty}(S)) \leq 3.$$                
\end{thm}

Next, we consider (a modification of) $\mathcal{G}_0(S)$, which may be regarded as an analog of the \textit{pants complex} \cite{DM} for infinite-type surfaces. Here, the vertices of $\mathcal{G}_0(S)$ are pants decompositions of $S$, and 
two vertices span an edge when they differ by exactly one elementary move, or by infinitely many of them performed simultaneously on pairwise disjoint subsurfaces `sufficiently far' (see Definition \ref{G0} for a formal description). We will see in Section \ref{S3}  that the graph $\mathcal{G}_0(S)$ is connected. 

In \cite{DM}, Margalit proved that the group of simplicial automorphisms of the pants complex coincides with the extended mapping class group. Thus, a natural question is whether the same holds for the graph $\mathcal{G}_0(S)$ introduced above. Our next result proves that this is indeed the case, provided one considers only {\em label-preserving} automorphisms. More concretely, we will say that an edge is a 1-edge (resp. $\infty$-edge) if its endpoints differ by one (resp. infinitely many) elementary move. Consider the subgroup $\Aut_L(\mathcal{G}_0(S))$ of $\Aut(\mathcal{G}_0(S))$ whose elements preserve the labelling of $\mathcal{G}_0(S)$. We will prove: 

\begin{thm}\label{MT2}
Let $S$ be an orientable connected surface of infinite-type. Then
$\Aut_L( \mathcal{G}_0(S))$ is isomorphic to $\Mod^\pm(S)$, the extended mapping class group of $S$.
\end{thm}

The plan of the paper is as follows. In Section \ref{S2} we will explain the classification of  infinite-type surfaces, plus some basic facts about multicurves and hyperbolic metrics on surfaces. 
Section \ref{S3} is devoted to the definition of the graphs $\mathcal{G}_{0}(S)$ and $\mathcal{G}_{\infty}(S)$. Finally, in Sections \ref{S4} and \ref{S5} we prove Theorems \ref{MT} and \ref{MT2}, respectively. 

\section{Surfaces, curves and hyperbolic metrics}\label{S2}
In this section, we introduce the classification theorems for finite and infinite-type surfaces and give some basic definitions about curves, multicurves and hyperbolic metrics.

Let $S$ be a compact connected orientable surface of finite topological type, i.e., with finitely generated fundamental group. Then $S \simeq S_{g,b}$, where $g$ is the genus, and $b$ the number of boundary components. The classification of such surfaces is well known: two finite-type connected orientable surfaces $S_{g,b}$ and $S'_{g',b'}$ are homeomorphic if and only if $g = g'$ and $b = b'$. If $S$ is an infinite-type surface then the homeomorphism type of $S$ is uniquely determined by its genus, its number of boundary components and its \textit{space of ends}. We refer the reader to \cite{IR} and \cite{PM} for a full discussion on the space of ends of a surface. From now on, we will assume that $S$ has no punctures or planar ends.

We briefly proceed to give some basic definitions about curves, multicurves and hyperbolic metrics. These objects are thoroughly covered in \cite{FM}.

Let $S$ be a finite or infinite-type connected orientable surface. A \textit{closed curve} $\alpha \subset S$ is a continuous map $\alpha: S^1 \rightarrow S$. To simplify, we denote $\alpha$ as the image of $S^1$ in $S$ by the homonymous map. In addition, we only consider the isotopy classes of closed curves in $S$ so by an abuse of notation, $\alpha = [\alpha]$ will be taken as a representative of all curves isotopic to $\alpha$. We will say that a curve is \textit{essential} if it is not isotopic to a point or a boundary component; and it is \textit{simple} if it may be realized without self-intersections. 

The \textit{geometric intersection number} $i(\alpha,\beta)$ between two curves $\alpha$ and $\beta$ is
\begin{align*}
i(\alpha,\beta) = \min\{\vert \alpha \cap \beta \vert, \alpha \in [\alpha], \beta \in [\beta] \}.
\end{align*}
Two curves $\alpha$ and $\beta$ are disjoint if and only if $i(\alpha,\beta) = 0$.

A \textit{multicurve} $\mu \subset S$ is a collection of essential simple closed curves $c_i$, such that they are pairwise disjoint and non-isotopic. Two multicurves $\mu$ and $\nu$ can be realized disjointly if:
\begin{align*}
\sum_{\alpha \in \mu, \beta \in \nu} i(\alpha,\beta) = 0.
\end{align*}

A \textit{pants decomposition} of $S$ is a locally finite multicurve $\mu \subset S$ which is maximal with respect to inclusion. The \textit{complexity} of a surface $S$ is the cardinality of a maximal multicurve of $S$, like a pants decomposition. For finite-type surfaces, the complexity is given by the following formula:
\begin{align*}
\kappa(S_{g,b}) = 3g-3+b.
\end{align*}
For infinite-type surfaces, the complexity is infinite.

It is a well known fact of hyperbolic geometry that every topological surface $S$ that has negative Euler characteristic admits a hyperbolic metric $\mathcal{H}$, which is a riemannian metric of constant curvature $-1$. We denote the pair $(S, \mathcal{H}) = S_{\mathcal{H}}$ as \textit{hyperbolic surface}. A metric allows us to compute the length of any curve $\alpha \subset S_\mathcal{H}$. In order to build a well defined length function, it must be only considered the length of the unique geodesic $\alpha_g$ isotopic to $\alpha$, which depends on the metric \cite[Proposition 1.3]{FM}. Thus the \textit{length} of $\alpha$ is defined as: $l(\alpha) = l(\alpha_g)$. We will focus our attention on the following family of hyperbolic metrics:

\begin{defi}\cite[\S 8]{AL}
Let $S_\mathcal{H}$ be a hyperbolic surface, we say that $S_\mathcal{H}$ is \textit{upper-bounded} with respect to some pants decomposition $\mathcal{P}$ if:
 \begin{align*}
 \exists M > 0 \ \mbox{such that} \ l(\alpha) < M \ \forall \alpha \in \mathcal{P}.
 \end{align*}
\end{defi}
We highlight that there are hyperbolic metrics which are not upper-bounded for any pants decomposition $\mathcal{P}$. We only need to find, for every infinite-type topological surface, a hyperbolic metric that is upper-bounded. This can be always done if the surface is built gluing pairs of pants, that is, surfaces homeomorphic to $S_{0,3}$, whose boundary components have length less than $M$. Then, the hyperbolic metric $\mathcal{H}$ is completely determined by its Fenchel-Nielsen coordinates \cite{AL}.

\section{Graphs $\mathcal{G}_k(S)$ and connectedness}\label{S3}

We proceed to describe two graphs developed by Fossas-Parlier \cite{FP}, which are constructed from multicurves on $S$: $\mathcal{G}_0(S)$ and $\mathcal{G}_\infty(S)$. From now on, let $S_{\mathcal{H}} = S$ be an infinite-type surface equipped with an upper-bounded hyperbolic metric. 

\begin{defi}\cite{FP}
$\mathcal{G}_{\infty}(S)$ is defined as the simplicial graph whose vertices are multicurves $\mu \subset S$ such that:
\begin{equation}
\sup\{\ell(\alpha_{\mu}) \vert \alpha_{\mu} \in \mu \ \mbox{is a connected component of} \ \mu \} < \infty.
\end{equation}
\begin{equation}
\sup\{\kappa(\Gamma) \vert \Gamma \ \mbox{is a connected component of} \ S \backslash \mu \} < \infty.
\end{equation}
Two vertices span an edge if they can be realized disjointly on $S$.
\end{defi}

We give some necessary definitions before defining $\mathcal{G}_{0}(S)$, which will be a slight modification of the one described in \cite{FP}.

\begin{defi} \cite[\S 2]{AR}
A multicurve $\mu \subset S$ has \textit{deficiency} $n$ if $\kappa(S \backslash \mu) = n$, that is, there exists a collection of simple closed curves $\{\alpha_1, \alpha_2,..., \alpha_n\}$ such that $\mu \sqcup \{\alpha_1, \alpha_2,..., \alpha_n\}$ is a pants decomposition of $S$.
\end{defi}

\begin{defi}\label{EM}
Two pants decompositions $X_1$ and $X_2$ differ by an \textit{elementary move} if there exists a deficiency-1 multicurve $\mu$ such that $X_1 = \mu \sqcup \alpha$ and $X_2 = \mu \sqcup \beta$, where $\alpha$ and $\beta$ intersect minimally. We say that $\alpha$ and $\beta$ \textit{intersect minimally} if $i(\alpha, \beta) = 1$ when the minimal subsurface containing them is $S_{1,1}$ and $i(\alpha, \beta) = 2$ when it is $S_{0,4}$.
\end{defi}

Note that $S_{1,1}$ and $S_{0,4}$ are the only two options because they are the unique surfaces of complexity one. In addition, $X_1 \cap X_2 = \mu$, since both pants decompositions share the same deficiency-1 multicurve.

\begin{defi}
Let $X$ be a pants decomposition of $S$ and let $\alpha$, $\beta$ two non-isotopic simple closed curves in $X$. If $\alpha$ and $\beta$ are the boundary of some pair of pants in $X$, then $\alpha$ and $\beta$ are \textit{sisters}. If they are not, but they belong to two different pair of pants that share any boundary component, they are \textit{contiguous}. Otherwise, they are \textit{far}. 
\end{defi} 

It is straightforward to check that the previous properties are invariant by elementary moves: let $\mu$ be a deficiency-1 curve and $X_1 = \mu \sqcup \alpha$, $X_2 = \mu \sqcup \beta$ two pants decompositions which differ by one elementary move. Then $\forall \gamma \subset \mu$, $\gamma$ and $\alpha$ are sisters (resp. contiguous or far) if and only if $\gamma$ and $\beta$ are sisters (resp. contiguous or far).

\begin{defi}
\label{G0} We define $\mathcal{G}_{0}(S)$ as the simplicial graph whose vertices are pants decompositions of $S$. Two vertices $v$ and $w$ span an edge if the pants decompositions associated to $v$ and $w$ differ by exactly one elementary move or by a countable subset of elementary moves replacing pairwise far curves.
\end{defi}

The difference between the original graph defined in \cite{FP} and this one is that in the former, simultaneously elementary moves between contiguous curves are allowed. Note that $\mathcal{G}_{0}(S)$ can be seen as the analog of the pants complex $P(S)$ for infinite-type surfaces, since it is connected. In addition, the graphs $\mathcal{G}_{\infty}(S)$ and $\mathcal{G}_{0}(S)$ are non-empty since the associated pants decomposition $\mathcal{P}$ of the upper-bounded hyperbolic metric $\mathcal{H}$ is a vertex of both. 

\begin{prop}
$\mathcal{G}_{\infty}(S)$ and $\mathcal{G}_{0}(S)$ are both connected.
\end{prop}

\begin{proof}
The cases $k=\infty$ and $k = 0$ for the original graph $\mathcal{G}_{0}(S)$ defined in \cite{FP}, are proven in \cite[Theorem 3.2]{FP}. We only need to prove the proposition for our modified $\mathcal{G}_{0}(S)$. However, it is always possible to decompose an edge corresponding to infinitely many elementary moves performed simultaneously on pairwise far or contiguous curves into a finite sequence of infinitely many elementary moves performed simultaneously only on pairwise far curves, since every curve of a pants decomposition has a finite number of contiguous curves. Therefore, our version of $\mathcal{G}_{0}(S)$ is also connected.
\end{proof}

Finally, in \cite{FP} it is shown that $\rank(\mathcal{G}_{0}(S)) = \infty$ \cite[Theorem 4.4]{FP}. Thus $\diam(\mathcal{G}_{0}(S)) = \infty$. In the next section we will prove that $\diam(\mathcal{G}_{\infty}(S)) < \infty$ for every infinite-type surface, generalizing a result of the aforementioned paper.

\section{The diameter of $\mathcal{G}_{\infty}(S)$}\label{S4}

The goal of this section is to prove Theorem \ref{MT}. To do that, we will first subdivide $S$ in subsurfaces using a family of separating curves that will be called a system of $i$-levels $\Lambda_i(S)$. Then, for every two vertices $\mu, \nu \in \mathcal{G}_{\infty}(S)$ we will build two vertices $v', w' \in \mathcal{G}_{\infty}(S)$ such that $\mu, v',w',\nu$ is a path in $\mathcal{G}_{\infty}(S)$.

\begin{defi}
Let $\alpha \subset S$ be an essential simple closed curve. Then $\alpha$ is a \textit{torus curve} if $S - \alpha = S' \sqcup S''$, where $S'$ is homeomorphic to a torus with one boundary component.
\end{defi}
Consider the maximal set $\mathcal{T}$ of torus curves in $S$. Then $S - \mathcal{T} \simeq S_0 \cup \bigsqcup_{\alpha \in \mathcal{T}} S_{1,1},$ where $S_0$ has genus zero and (possibly) boundary components. Let $\mathcal{P}$ be a pants decomposition of $S_0$ such that $S_0$ is upper bounded with respect to it. We define $\Gamma(S_0)$ as the graph whose vertices are in correspondence one-to-one with every pair of pants and every boundary component of $S_0$. Two vertices span and edge in $\Gamma(S_0)$ if they represent two pair of pants that share a curve in $\mathcal{P}$, or if they represent a boundary component of $S_0$ and the pair of pants to which it belongs. Observe that $\Gamma(S_0)$ is an infinite tree, because $S_0$ is a surface of genus zero.  

Now, the set of vertices of $\Gamma(S_0)$ can be divided into two families: $3$-degree vertices representing pair of pants, and $1$-degree vertices representing boundary components of $S_0$. We denote the first family by $V_3(\Gamma(S_0))$ or simply $V_3$. Equivalently we divide the set of edges of $\Gamma(S_0)$ into the subset that link two $3$-degree vertices and the ones which link a $3$-degree vertex with a $1$-degree vertex. We denote the first family by $E_3(\Gamma(S_0)) = E_3$. As a help for the reader the construction of $\Gamma(S_0)$ is depicted in Figure \ref{cons}: \begin{figure}[H]
\centering
\includegraphics[width=1\textwidth]{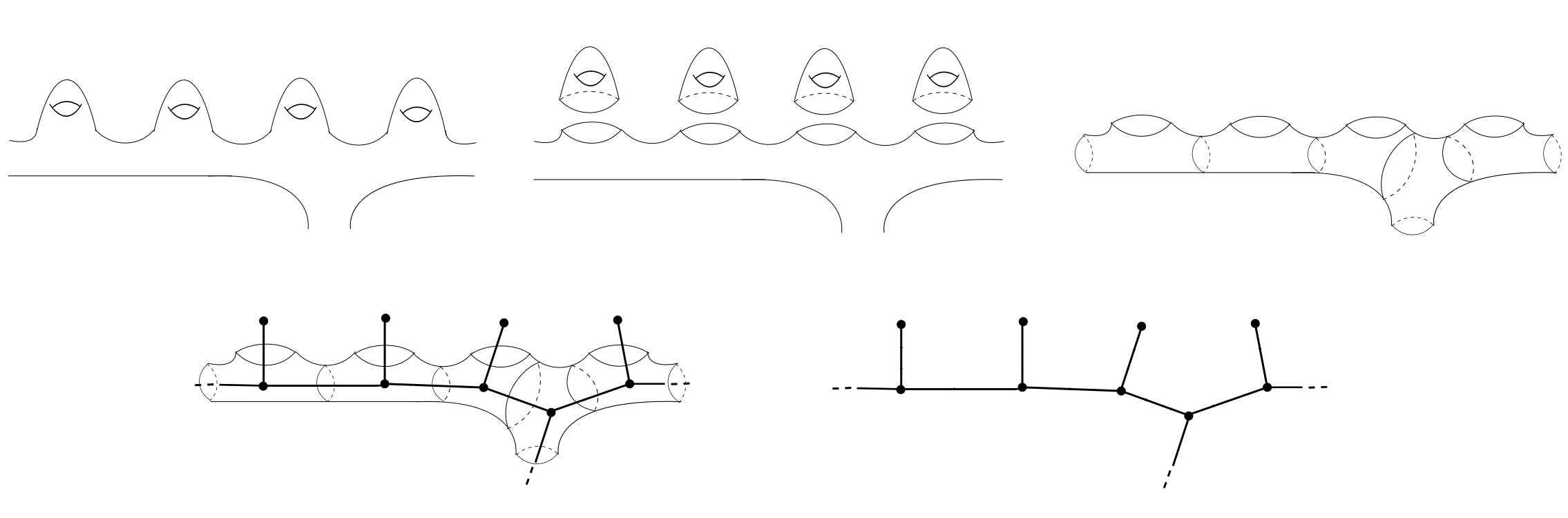} 
\caption{}
\label{cons}
\end{figure} We take $\Gamma_3 = V_3 \cup E_3$. Note that $\Gamma_3$ is also an infinite tree because it is a connected subgraph of $\Gamma$. Consider, for every edge $e \in E_3$, the simple closed curve $\gamma \in \mathcal{P}$ which is the shared boundary between the two pairs of pants associated to the endpoints of $e$. By an abuse of notation we identify $e$ with $\gamma$. This curve is always separating in $S_0$ and, in addition, in $S$. Let $d_{\Gamma_{3}}$ be the distance in $\Gamma_3$, which is the restriction of the usual distance associated to $\Gamma(S_0)$. We consider the following family of curves: starting in an arbitrary edge $\gamma^0 \in E_3$, $\gamma^0$ will be the \textit{$0$-level} of $S$, $ \Lambda_0(S)$. Then we define:
\begin{center}
$ \Lambda_i(S) = i-level \ $of$ \ S  = \{\gamma \in E_3: d_{\Gamma_3}(\gamma, \gamma^0)= i\} = \{\gamma^i_k\}_{k \in N(i)}$,
\end{center}
where $N(i) < \infty$ is the number of curves in $\Lambda_i(S)$. Note that every edge in $E_3$ belongs to $\Lambda_i$ for a unique $i \in \N$, because $\Gamma_3(S_0)$ is a tree. Figure \ref{levels} shows an example of $i$-levels: 
\begin{figure}[H]
\centering
\includegraphics[width=0.45\textwidth]{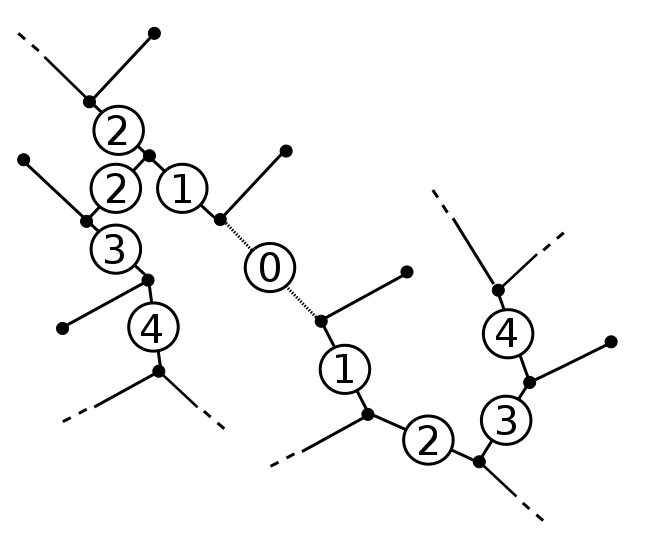} 
\caption{}
\label{levels}
\end{figure}

\begin{proof}[Proof of Theorem \ref{MT}]
For $\mu, \nu \in \mathcal{G}_{\infty}(S)$, take two pants decompositions $v,w \in \mathcal{G}_{\infty}(S)$, such that $\mu \subset v$ y $\nu \subset w$. Consider $L  = \max\{L_v,L_w\}$, where:
\begin{align*}
L_v & = \sup\{\ell(\alpha_{v}) \vert \ \alpha_{v} \ \mbox{is a curve of} \ v \}, \\
L_w & = \sup\{\ell(\alpha_{w}) \vert \ \alpha_{w} \ \mbox{is a curve of} \ w \}.
\end{align*}
If we choose $v$ and any curve $\gamma^i_k \in \mathcal{P}$, we define $S_v(\gamma^i_k)$ as the minimal finite subsurface with respect to inclusion, containing all the curves of $v$ that intersect $\gamma^i_k$. Observe that $S \backslash S_v(\gamma^i_k)$ is a set of at most two connected subsurfaces, because $S_v(\gamma^i_k)$ contains $\gamma^i_k$, which is separating in $S$. In addition $\kappa(S_v(\gamma^i_k)) < \infty$, by the collar lemma \cite[Lemma 13.6]{FM}. Otherwise, $\gamma^i_k$ would be crossed by an infinite number of curves, every of them with a collar neighbourhood of a fixed area. So $\gamma^i_k$ would have infinite length, which is a contradiction with the upper-bounded metric hypothesis. We define $S_w(\gamma^i_k)$ in the same way for the pants decomposition $w$.

Now, note that if we take $d_S(\gamma^i_k,\gamma^{i'}_{k'}) > L$, then $S_x(\gamma^i_k)$ and $S_y(\gamma^{i'}_{k'})$ are disjoint for every $x,y \in \{v,w\}$. Indeed, if some curve belongs to $S_x(\gamma^i_k)$ and $S_y(\gamma^{i'}_{k'})$ simultaneously, it cuts both $\gamma^i_k$ and $\gamma^{i'}_{k'}$, so it has length greater than $L$, which cannot be possible. Consider a subsequence of levels $\Lambda_{i_j}$, $j \in \N$, such that: 
\begin{align}
\inf \{ d_S(\gamma^{i_j}_{k_j},\mathcal{N}_{i_j} )\} > L  , \ \forall \ 1 \leq k_j \leq N(i_j),\\
\sup \{ d_S(\gamma^{i_j}_{k_j},\mathcal{N}_{i_j} )\} <\infty , \ \forall \ 1 \leq k_j \leq N(i_j),
\end{align}
where $\mathcal{N}_{i_j}$ is the subset of curves of $\Lambda_{i_{j+1}}$ whose representing edges are the closest to $\gamma^{i_j}_{k_j}$, that is, there is no other edges of  $\Lambda_{i_j}$ between them.
We point out that it is always possible to find a subset of levels which fulfils assumptions $(3)$ and $(4)$ because of the upper-bounded hyperbolic metric hypothesis. As in Section \ref{S2}, we construct every infinite-type surface from a pants decomposition $\mathcal{P}$ which satisfies $(3)$ and $(4)$ simultaneously. 

For even $j \in \mathbb{N}$, let $v' \subset v$ be the multicurve:
\begin{align*}
v' = \Bigg \{ \alpha \in v: \alpha \in \bigcup_{j \in 2\mathbb{Z}} \ \bigcup_{1 \leq k_j \leq \vert \Lambda_{i_j} \vert} S_v(\gamma^{i_j}_{k_j}) \Bigg  \},
\end{align*}

and $w' \subset w$ the multicurve obtained for odd $j \in \mathbb{N}$. Note that $v'$ is a vertex of $\mathcal{G}_{\infty}(S)$ even if some of the subsurfaces $S_v(\gamma^{i_j}_{k_j})$ need not be disjoint for a fixed level $\Lambda_{i_j}$. This could happen when the subsurfaces are defined by curves in the same level such that the distance between them is less than $L$. If this is the case, the distance between the closest curves of two consecutive levels of the subsequence is bounded above by $(4)$, and so is the complexity of every connected component of $S \backslash v'$. Thus $v'$ is a vertex of $\mathcal{G}_{\infty}(S)$. By a similar argument, $w'$ also is a vertex of $\mathcal{G}_{\infty}(S)$.

Now, $\mu$ and $v'$ span an edge by construction because $v'$ and $\mu$ are subsets of the same pants decomposition $v$, so they can be realized disjointly on $S$. The same occurs with $\nu$ and $w'$ with respect to $w$. Finally, $v'$ and $w'$ span an edge because every component of $v'$ is disjoint from any component of $w'$, as they belong to subsurfaces that are distance at least $L$.\end{proof}

\section{Proof of Theorem \ref{MT2}}\label{S5}

In this section we prove Theorem \ref{MT2}. First we study some loops in $\mathcal{G}_0(S)$ which are essential in the proof: $\infty$-alternating squares. Finally, we define the map $\phi: \Aut_L( \mathcal{G}_0(S)) \rightarrow \Aut(C(S))$ and prove that it is an isomorphism. Since $\Aut(C(S)) \simeq \Mod^\pm (S)$, as shown in \cite{BDR, HF}, the main result follows. The most difficult point is to prove that $\phi$ is well defined, so we will use the previous loops to achieve that. This section is heavily inspired in the proof for finite-type surfaces, due to Margalit \cite{DM}.

\begin{defi}
A \textit{path} in $\mathcal{G}_0(S)$ is a set $X_1,...,X_{k+1}$ of vertices such that $X_i$ and $X_{i+1}$ span an edge $\forall i\in\{ 1,...,k \}$. If $X_1 = X_{k+1}$, we call it a \textit{loop} and omit the repeated vertex. 

We define a 1-\textit{triangle} $\mathcal{T} = \{X_1,X_2,X_3\}$ as a loop of three vertices which span a 1-edge pairwise \cite[\S 3.1]{DM}.
\end{defi}

\begin{lemma}[Characterization of 1-triangles] Let $T = \{X_1,X_2,X_3\}$ be a loop in $\mathcal{G}_0(S)$. Then $T$ is a 1-triangle if and only if there exists a deficiency-1 multicurve $\mu$ such that $X_i = \mu \sqcup \alpha_i$, where $\alpha_i$ and $\alpha_j$ intersect minimally pairwise.
\end{lemma}

\begin{proof}
$\Rightarrow$) By definition of elementary move (Definition \ref{EM}), $X_1 \cap X_2 = \mu$, where $\mu$ is a curve of deficiency 1. Since $X_3$ span a 1-edge with both $X_1$ and $X_2$, $X_1 \cap X_2 = X_1 \cap X_2 \cap X_3 = \mu$. So $X_i = \mu \sqcup \alpha_i$, where $\alpha_i$, $\alpha_j$ intersect minimally for $i \neq j$. 

$\Leftarrow$) Every two different vertices of $T$ differ by an elementary move. Thus $T$ is a 1-triangle. \end{proof}

\begin{corollary}[Transitivity]\label{TR}
If two 1-triangles $\mathcal{T}$ and $\mathcal{T}'$ share a 1-edge, then there exists a deficiency-1 multicurve $\mu$ such that every vertex of both $\mathcal{T}$ and $\mathcal{T}'$ contains $\mu$.
\end{corollary}
\begin{proof}
Suppose that $\mathcal{T} = \{X_1,X_2,X\}$ and $\mathcal{T}'= \{X_1,X_2,X'\}$. Since $X$ and $X'$ span a 1-edge with both $X_1$ and $X_2$, $X \cap X_1 \cap X_2 = \mu = X' \cap X_1 \cap X_2$.
\end{proof}
\begin{defi} \cite[\S 2]{RM}
The standard \textit{Farey graph} is the simplicial graph whose vertices are:
\begin{align*}
\bigg \{ \frac{p}{q} \ : \ p,q \in \Z, \frac{p}{q} \ \mbox{is irreducible} \ \bigg\} \cup  \bigg \{ \frac{1}{0} = \infty  \bigg\},
\end{align*} 
where two vertices $\displaystyle \frac{p}{q}$ and $\displaystyle \frac{s}{t} $ span an edge if $\vert pt-qs \vert = 1$.
\end{defi}
The standard Farey graph can be realized as an ideal triangulation of the hyperbolic disk, as it is depicted in Figure \ref{DFG}:
\begin{figure}[H]
\labellist
\pinlabel $1$ at 225 430
\pinlabel $-1$ at 225 15
\pinlabel $0$ at 425 220
\pinlabel $\infty$ at 25 220
\pinlabel $2$ at 80 380
\pinlabel $3$ at 40 300
\pinlabel $\frac{1}{3}$ at 410 300
\pinlabel $\frac{1}{2}$ at 370 380
\pinlabel $\frac{3}{2}$ at 150 420
\pinlabel $\frac{2}{3}$ at 300 420
\pinlabel $-\frac{1}{2}$ at 370 70
\pinlabel $-2$ at 80 70
\endlabellist
\centering
\includegraphics[width=0.34\textwidth]{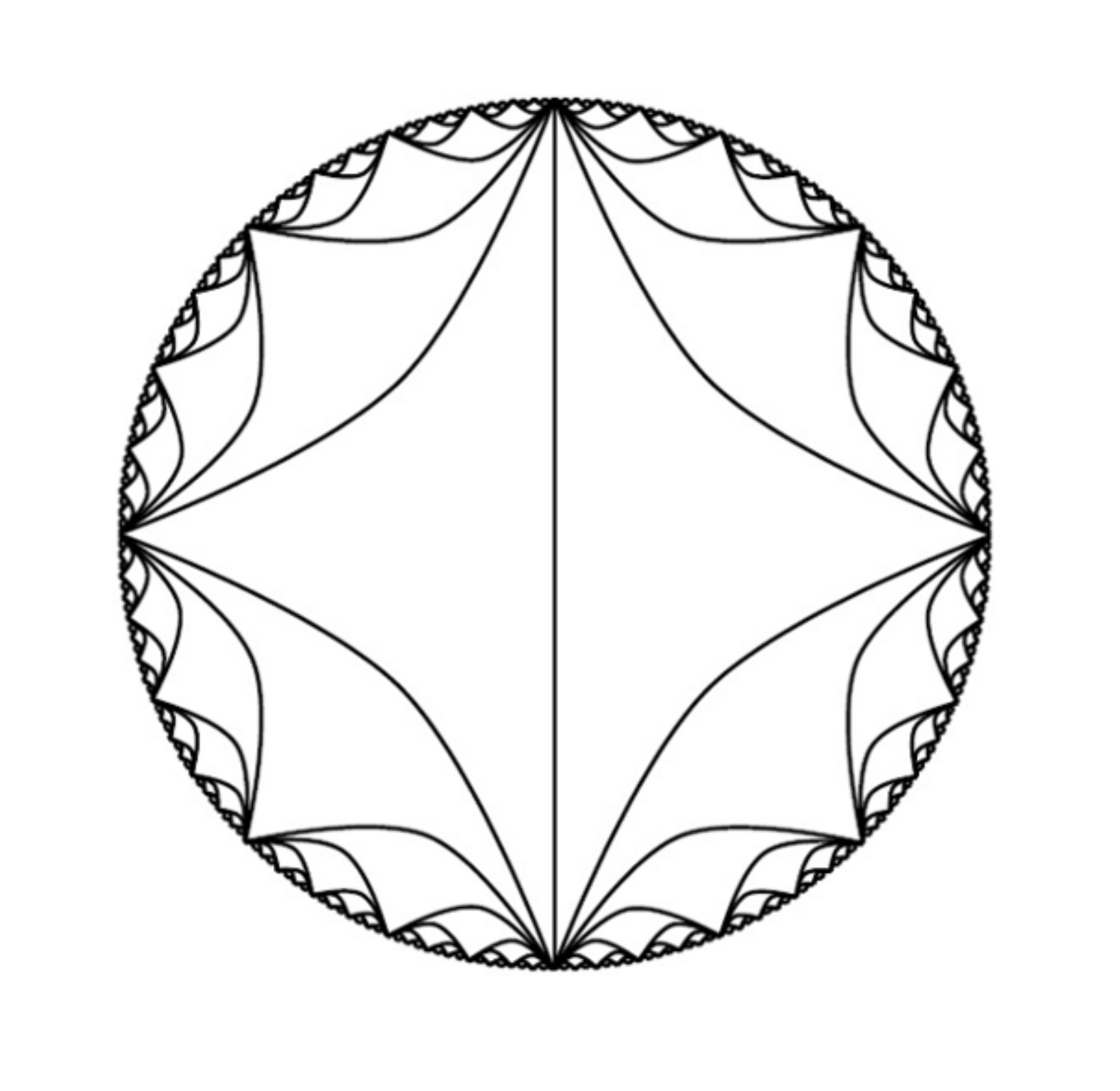}
\caption{}
\label{DFG}
\end{figure}

\begin{defi}
A \textit{1-Farey graph} is a subgraph of $\mathcal{G}_0(S)$ which is isomorphic to the standard Farey graph and whose edges are all 1-edges. 
\end{defi}

Note that for every pair of vertices in a 1-Farey graph which span an edge, there are always two different 1-triangles which contain them. 

\begin{lemma}[Transitivity for 1-Farey graphs]\label{FG}
For every two vertices $X$, $X'$ in a 1-Farey graph there is a sequence of 1-triangles $\{\mathcal{T}_1,...,\mathcal{T}_n\}$ such that $X \in \mathcal{T}_1$, $X' \in \mathcal{T}_n$ and every two consecutive 1-triangles $\mathcal{T}_i$, $\mathcal{T}_{i+1}$ have two vertices in common.
\end{lemma}
\begin{proof} 
Since the dual of the standard Farey graph is an infinite tree with all vertices of degree 3, we take the vertices $T_1$ and $T_n$ whose duals are the triangles $\mathcal{T}_1$ and $\mathcal{T}_n$ respectively in the 1-Farey graph. Then, a path in the tree joining $T_1$ and $T_n$ has the desired sequence $\{\mathcal{T}_1,...,\mathcal{T}_n\}$ as dual.
\end{proof}
By Lemma \ref{FG} and Corollary \ref{TR}, there exists a deficiency-1 multicurve $\mu$ such that every pants decomposition of a 1-Farey graph contains $\mu$.

\begin{defi} 
For a multicurve $\mu$, we define $P_\mu$ as the subgraph of $\mathcal{G}_0(S)$ whose vertices are pants decompositions which contain $\mu$. 
\end{defi}

For any pair of multicurves $\mu, \nu \subset S$: \begin{center}
$P_{\mu} \cap P_{\nu} = \begin{cases} P_{\mu \cup \nu} & \mbox{if} \ \mu \cup \nu \ \mbox{is a multicurve}, \\ \emptyset & \mbox{otherwise.} \end{cases}$
\end{center}

\begin{lemma}[Characterization of 1-Farey graphs in $\mathcal{G}_0(S)$]\label{CFG}
Let $F \subset \mathcal{G}_0(S)$ be a subgraph. Then $F$ is a 1-Farey graph if and only if there exists a multicurve $\mu$ of deficiency 1 such that $F = P_{\mu}$.
\end{lemma}

\begin{proof}
 $\Leftarrow)$ Since all vertices in $F$ share a deficiency-1 multicurve $\mu$, every vertex $X_i \in F$ has the form $X_i = \mu \sqcup \alpha_i$. There is an isomorphism from $F$ to $P(S \backslash \mu)$ which takes $X_i = \mu \sqcup \alpha_i$ to $\alpha_i$, where the unique subsurface of positive complexity in $S \backslash \mu$ is homeomorphic to $S_{0,4}$ or $S_{1,1}$. The pants complex of both surfaces is isomorphic to the standard Farey graph \cite[\S 3]{MI}. 

$\Rightarrow)$ Let $F$ be a 1-Farey graph. By Lemma \ref{FG} it follows that $V(F) \subset P_{\mu}$. Since $F$ is isomorphic to $P(S \backslash \mu)$, then $V(F) = P_{\mu}$.
\end{proof}

\begin{lemma}\label{int}
Two different 1-Farey graphs intersect at most in one vertex.
\end{lemma}
\begin{proof}
Let $F$ and $F'$ be two different 1-Farey graphs. By Lemma \ref{CFG}, $F = P_{\mu}$ and  $F'= P_{\nu}$. If $F$ and $F'$ intersect, then $F \cap F' = P_{\mu} \cap P_{\nu} = P_{\mu \cup \nu}.$ As $\mu$ and $\nu$ are different deficiency-1 multicurves, $P_{\mu \cup \nu}$ is non-empty when $\mu \cup \nu$ is a pants decomposition.
\end{proof}

Note that, for any 1-Farey graph $F$, the deficiency-1 multicurve $\mu$ such that $F = P_\mu$ can be identified by intersecting two different vertices of $F$.

\begin{defi}
A \textit{marked 1-Farey graph} is a pair $(F, X)$ where $F$ is a 1-Farey graph and $X$ is one of its vertices.
\end{defi} 

Note that for every marked 1-Farey graph $(F=P_{\mu},X)$ there exists a single curve $\alpha \in X$ such that $\alpha \not\in P_\mu$. We will say that $(F,X)$ \textit{represents} $\alpha$. We will also consider other loops in $\mathcal{G}_0(S)$. Alternating loops, which are defined below, are introduced in \cite[\S 4]{DM}. First we need a definition:

\begin{defi}
A loop $\mathcal{L} \subset \mathcal{G}_0(S)$ has deficiency $k$ if the intersection of all its vertices is a deficiency-$k$ multicurve.
\end{defi}
For instance, a 1-triangle is a deficiency-1 loop.
\begin{defi}\cite{DM}
An \textit{alternating loop} $\mathcal{AL} \subset \mathcal{G}_0(S)$ is a deficiency-$2$ loop $\{X_1,X_2,...,X_k\}$, $k \in \{4,5,6\}$,  whose edges are all 1-edges and such that there is no 1-Farey graph in $\mathcal{G}_0(S)$ containing any three consecutive vertices $X_i,X_{i+1},X_{i+2}$.
\end{defi}

Observe that in an alternating loop $X_1 \cap X_2 \cap...\cap X_n = X_i \cap X_{i+1} \cap X_{i+2} \ \forall i$.

\begin{remark}\label{ALP} \cite[\S 5]{DM}
Any alternating loop $\mathcal{AL}$ is preserved by the action of $A \in \Aut_L(\mathcal{G}_0(S))$, since it is defined only in simplicial terms.
\end{remark}

\begin{defi}
An \textit{$\infty$-alternating square} $\mathcal{S} \subset \mathcal{G}_0(S)$ is a loop of length $4$ with the following structure, up to cyclic reordering: $$X_1 \oset{1}{-} X_2  \oset{$\infty$}{-} X_3 \oset{1}{-} X_4 \oset{$\infty$}{-} X_1,$$
where $X_1$ (resp. $X_2$) do not span an $\infty$-edge with $X_3$  (resp. $X_4$).
\end{defi}

Let $\mathcal{S}$ be an $\infty$-alternating square. Without loss of generality: $$X_1 = \{\alpha, \sigma_1,..., \sigma_m, \kappa_1,...,\kappa_n, \lambda_1,...\},$$ $$X_2 = \{\alpha', \sigma_1,..., \sigma_m, \kappa_1,...,\kappa_n, \lambda_1,...\},$$
where $\{\sigma_i\}_{i = 1}^m$ are the sister curves of both $\alpha$ and $\alpha'$, $\{\kappa_i\}_{i = 1}^n$ the contiguous curves and $\{\lambda_i\}_{i = 1}^\infty$ the far curves. Note that $1 \leq m \leq 4 $ and $1 \leq n \leq 8$. Now, since $X_2X_3$ is an $\infty$-edge and $X_1$, $X_3$ do not span an $\infty$-edge, the elementary moves are not performed only on far curves. Furthermore, no elementary move is performed on a sister curve because $\mathcal{S}$ has length 4. So: 
$$X_3 = \{\alpha', \sigma_1,..., \sigma_m, \kappa'_1,...,\kappa'_j,\kappa_{j+1},...,\kappa_n, \lambda_1,\lambda'_2,\lambda_3,\lambda'_4,...\},$$
where the infinitely many elementary moves are performed, without loss of generality on $\lambda_{2k}$, $k \in \N$ and some contiguous curves. Observe that all these simultaneous elementary moves must be performed on curves which are pairwise far by the definition of $\infty$-edge. Regarding the 1-edge $X_3X_4$, if $X_4$ contains any $\kappa_i$ or $\lambda_{2k}$, then $X_2$ and $X_4$ span an $\infty$-edge, so it must contain $\alpha$:
$$X_4 = \{\alpha, \sigma_1,..., \sigma_m, \kappa'_1,...,\kappa'_j,\kappa_{j+1},...,\kappa_n, \lambda_1,\lambda'_2,\lambda_3,\lambda'_4,...\}.$$ 
Therefore: $X_1 \cap ... \cap X_4 = X_i \cap X_{i+1} \cap X_{i+2} \ \forall i$.

\begin{remark}\label{SAL}
Any $\infty$-alternating square $\mathcal{S}$ is preserved by the action of $A \in \Aut_L(\mathcal{G}_0(S))$, since it is defined only in simplicial terms.
\end{remark}

\begin{proof}[Proof of Theorem \ref{MT2}]
We proceed to define the isomorphism of Theorem \ref{MT2}. Consider $A \in \Aut_L(\mathcal{G}_0(S))$, so $A$ takes marked 1-Farey graphs to marked 1-Farey graphs. Let $(F,X)$ be an arbitrary marked 1-Farey graph. Then $(F,X)$ represents some curve $v \in C(S)$. Let $A((F,X))$ be the image of $(F,X)$ by $A$, then $A((F,X))$ represents another vertex $w$ in the curve complex. We define $w = \phi(A)(v)$; thus $\phi(A)$ is a map from $C(S)$ to itself. We will prove that $\phi(A)$ is an automorphism of the curve complex, and that $\phi$ is a well defined isomorphism.

\textbf{$\phi(A)$ is well defined}. Let $(F,X)$ and $(F',X')$ be two marked 1-Farey graphs representing the same vertex $v \in C(S)$. We need to prove that $A((F,X))$ and $A((F',X'))$ also represent the same vertex $\phi(A)(v)$. This is true when $X$ and $X'$ differ by a finite number of elementary moves  \cite[\S 5]{DM}. The idea of \cite{DM} is to find a sequence of vertices $W,X,X',Y$ which is contained in an alternating loop $\mathcal{AL}$, where $W,X \in (F,X)$ and $X',Y \in (F',X')$. By the properties of an alternating loop, $v = W \cap X \cap X' = X \cap X' \cap Y$. Now, by Remark \ref{ALP}, $\mathcal{AL}$ is preserved by the action of $A \in \Aut_L(\mathcal{G}_0(S))$, so the sequence $A(W),A(X),A(X'),A(Y)$ is also contained in an alternating loop. Therefore $A(W) \cap A(X) \cap A(X') = A(X) \cap A(X') \cap A(Y)$, so $A((F,X))$ and $A((F',X'))$ represent the same vertex $\phi(A)(v)$. 

When $X$ and $X'$ span an $\infty$-edge, the previous idea can be applied finding a sequence contained in an $\infty$-alternating square $\mathcal{S}$, as we will see. Once this is done, Remark \ref{SAL} will again give us the desired result. We refer the reader to Figure \ref{esq} as a help for the following proof. Let $X = \{\alpha', \mu\}$ and $X'  =  \{\alpha', \mu'\}$ where $F = P_\mu$ and $F' = P_{\mu'}$. Suppose that at least one elementary move of $XX'$ is performed on a contiguous curve of $\alpha'$. Then $X$ and $X'$ have the form: 
\begin{align*}
X  & = \{\alpha', \sigma_1,..., \sigma_m, \kappa_1,...,\kappa_n, \lambda_1,...\}, \\
X' & = \{\alpha', \sigma'_1, \sigma_2,..., \sigma_m, \kappa'_1,...\kappa'_j,\kappa_{j+1},...,\kappa_n, \lambda_1,\lambda'_2,\lambda_3,\lambda'_4,...\},
\end{align*}
Note that two different sister curves of $\alpha$ are respectively sister or contiguous, so at most one sister curve can be replaced in an $\infty$-edge.  Now consider $X_m$, where:
$$X_m = \{\alpha', \mu_m\} = \{\alpha', \sigma'_1,\sigma_2,..., \sigma_m, \kappa_1,...,\kappa_n, \lambda_1,\lambda_2,...\},$$
with $F_m = P_{\mu_m}$. As $X$ and $X_m$ differ by one elementary move, we use the the result in \cite{DM} to conclude that $A((F,X))$ and $A((F_m,X_m))$ represent the same vertex in $C(S)$. Finally, we proceed with the contiguous and far curves of $\alpha'$. Consider the following $\infty$-alternating square: $$\mathcal{S}: W \oset{1}{-} X_{m}  \oset{$\infty$}{-} X' \oset{1}{-} Y \oset{$\infty$}{-} W,$$ where:
\begin{align*}
W & =\{\alpha, \sigma'_1,\sigma_2,..., \sigma_m,\kappa_1,...\kappa_{j},\kappa_{j+1},...,\kappa_n, \lambda_1,\lambda_2,\lambda_3,\lambda_4,...\},\\
X_{m} &=\{\alpha', \sigma'_1,\sigma_2,..., \sigma_m,\kappa_1,...\kappa_{j},\kappa_{j+1},...,\kappa_n, \lambda_1,\lambda_2,\lambda_3,\lambda_4,...\},\\
X' &= \{\alpha', \sigma'_1,\sigma_2,..., \sigma_m,\kappa'_1,...\kappa'_{j},\kappa_{j+1},...,\kappa_n, \lambda_1,\lambda'_2,\lambda_3,\lambda'_4,...\},\\
Y &= \{\alpha, \sigma'_1,\sigma_2,..., \sigma_m,\kappa'_1,...\kappa'_{j},\kappa_{j+1},...,\kappa_n, \lambda_1,\lambda'_2,\lambda_3,\lambda'_4,...\},
\end{align*} 
with $W,X_m \in (F_{m},X_{m})$ and $X',Y \in (F',X')$. We use Remark \ref{SAL} to deduce that $A((F_{m},X_{m}))$ and $A((F',X'))$ represent the same vertex in $C(S)$. Thus $A((F,X))$ and $A((F',X'))$ represent also the same vertex.

If no elementary move is performed on a contiguous curve of $\alpha'$, then $X'$ has the form: $$X' = \{\alpha', \sigma'_1,\sigma_2,..., \sigma_m, \kappa_1,...,\kappa_n, \lambda_1,\lambda'_2,\lambda_3,\lambda'_4,...\}.$$
In that case consider first:
$$X'' = \{\alpha', \sigma'_1,\sigma_2,..., \sigma_m, \kappa'_1,\kappa_2...,\kappa_n, \lambda_1,\lambda'_2,\lambda_3,\lambda'_4,...\}.$$  Observe that $\kappa'_1$ must be far from all $\lambda_{2k}$. As $X'$ and $X''$ differ by one elementary move, the marked 1-Farey graphs $(F',X')$ and $(F'',X'')$ represent the same curve in $C(S)$. Then proceed with $(F'',X'')$ in the same way as before.
\begin{figure}[H]
\centering
\labellist
\pinlabel $F$ at 50 500
\pinlabel $F_m$ at 40 290
\pinlabel $F'$ at 50 90
\pinlabel $X$ at 630 500
\pinlabel $X_m$ at 630 290
\pinlabel $X'$ at 630 90
\pinlabel $W$ at 410 280
\pinlabel $Y$ at 380 70
\pinlabel $1$ at 480 110
\pinlabel $1$ at 480 310
\pinlabel $1$ at 450 520
\pinlabel $1$ at 610 390
\pinlabel $\infty$ at 620 190
\pinlabel $\infty$ at 370 190
\endlabellist
\includegraphics[width=0.38\textwidth]{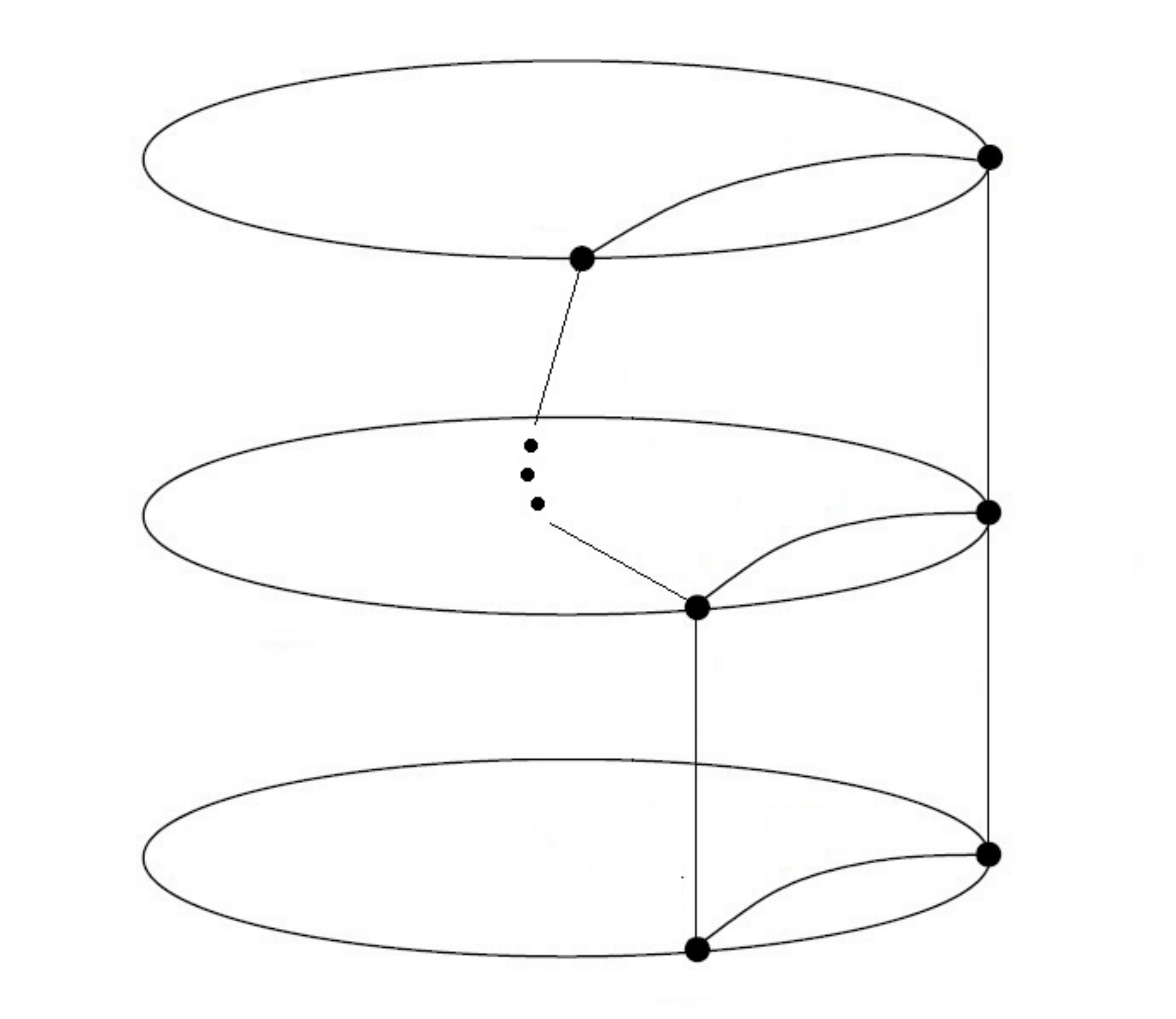}
\caption{}\label{esq}
\end{figure} \textbf{$\phi(A)$ is an automorphism}. We need to prove that two vertices $v,w \in C(S)$ span an edge if and only if $\phi(A)(v)$ and $\phi(A)(w)$ span an edge. For that purpose is enough to prove that $v$ and $w$ span an edge in $C(S)$ if and only if there are two different marked 1-Farey graphs $(F,X)$ and $(F',X)$, which intersect in $X$. This vertex is unique by Lemma \ref{int}. Suppose that $v, w$ are the curves $\alpha$ and $\beta$ respectively. We take a pants decomposition containing $\alpha$ and $\beta$: $X = \{ \alpha, \beta, \gamma_1, \gamma_2,... \}$, and consider the marked $1$-Farey graphs $(F,X)$ and $(F',X)$, where $F = P_{\mu}$ and $F' = P_{\nu}$ with $ \mu = \{\beta, \gamma_1, \gamma_2,... \}$ and $ \nu = \{ \alpha, \gamma_1, \gamma_2,... \}$. Since $A$ preserves the number of intersection points between marked 1-Farey graphs,  $A((F,X))$ and $A((F',X))$ intersect in the point $A(X) = \{\phi(A)(\alpha),\phi(A)(\beta),...\}$. Then, $A(X)$ is a pants decomposition if and only if $\phi(A)(\alpha)$ an $\phi(A)(\beta)$ are disjoint and non-isotopic, that is, they span and edge in $C(S)$. Thus $F$ and $F'$ intersect in $X$ if and only if $v$ and $w$ span an edge in $C(S)$.

\textbf{$\phi$ is an isomorphism}. First, we prove that $\phi(AB)(v) = \phi(A)\phi(B)(v)$ for every $v \in C(S)$. $\phi(AB)(v)$ is represented by $AB((F,X))$,  where $(F,X)$ represents $v$. Furthermore, $w = \phi(B)(v)$ is represented by $B((F,X))$ and hence $\phi(A)\phi(B)(v)$ is represented by $A((F',Y))$, where $(F',Y)$ is a marked $1$-Farey graph that represents $w$. If we choose $(F',Y)$ to be $B((F,X))$, then we have that $\phi(A)\phi(B)(v)$ is represented by $AB((F,X))$ and so $\phi$ is a homomorphism.

We now prove that if $\phi(A)$ is the identity in $\Aut(C(S))$, then $A$ is the identity in $\Aut_L(\mathcal{G}_0(S))$. For $X = \{ \alpha_1, \alpha_2, \alpha_3,... \}$  we choose the collection of marked $1$-Farey graphs $(F_i,P)$, where $F_i = P_{\mu_i}$, with \linebreak $\mu_i = \{ \alpha_1, \alpha_2,..., \alpha_{i-1}, \alpha_{i+1},....\}$, and $v_i$ is the curve $\alpha_i$. By construction, the intersection of all $F_i$ is exactly $X$. In addition, because $A(F)$ represents $\phi(A)(v)$ and $\phi(A)$ is the identity, it follows that the intersection of all $A(F_i) = F_i$ is $X$, so $A(X) = X$.
 
Finally, we will prove the surjectivity of $\phi$ by using the isomorphism $\eta: \Aut(C(S)) \rightarrow \Mod^{\pm}(S)$ for infinite-type surfaces \cite{BDR, HF}. For a vertex $v \in C(S)$, we define $A(v)$ as the action of $\eta (A)$ in $\alpha$. Similarly, for a pants decomposition $X = \{ \alpha_1, \alpha_2, \alpha_3,... \}$, we define $\psi(A)(X)$ as $\{\eta(A)(\alpha_1),\eta(A)(\alpha_2),... \}$: \begin{center}
$\begin{array}{ccccc}
  \Aut_L(\mathcal{G}_0(S)) & \overset{\psi}{\longleftarrow} &  \Aut(C(S)) &   \overset{\eta}{\longleftrightarrow} &\Mod^{\pm}(S)  \\
 \psi(A) & \longleftarrow & A &\longleftrightarrow& \eta(A)   \\
 (F,X) & \longleftrightarrow & v &  \longleftrightarrow & \alpha   \\
 \downarrow & &\downarrow && \downarrow  \\
 \psi(A)((F,X)) & \longleftrightarrow & A(v) & \longleftrightarrow &  \eta(A)(\alpha)  \\
\end{array}$
\end{center}
 It remains to prove that $\phi \circ \psi (A) = A$, which is equivalent to prove that $\psi(A)((F,X)) = (F',Y)$, where $(F',Y)$ represents the vertex $A(v) \in C(S)$. For $F= P_{\mu}$ where $\mu = \{\alpha_2, \alpha_3,...\}$, it follows that $F' = \psi(A)(F) = P_\nu$, where $ \nu = \{\eta(A)(\alpha_2),\eta(A)(\alpha_3),... \}$. The desired result follows from the fact that $X= \{ \alpha, \alpha_2, \alpha_3,... \}$, $\psi(A)(X) = \{\eta(A)(\alpha),\eta(A)(\alpha_2),\eta(A)(\alpha_3),... \}$ and $\eta(A)(\alpha)$ corresponds to $A(v)$.\end{proof}

\section*{Acknowledgements}
The author would like to thank his advisor, Javier Aramayona, for conversations and for suggesting the problem. He also wants to acknowledge financial support from the Spanish Ministry of Economy and Competitiveness, through the “Severo Ochoa Programme for Centres of Excellence in R\&D” (SEV-2015-0554) and the grant MTM2015-67781.
\label{Bibliography}

\bibliographystyle{abbrv} 
\bibliography{Bibliography} 

\begin{thebibliography}{10}

\bibitem{AL}
D.~Alessandrini, L.~Liu, A.~Papadopoulos, W.~Su, and Z.~Sun.
\newblock On {F}enchel-{N}ielsen coordinates on {T}eichmüller spaces of
  surfaces of infinite type.
\newblock {\em Comm. Anal. Geom.}, 20(2):369--394, (2012).

\bibitem{AR}
J.~Aramayona.
\newblock Simplicial embeddings between pants graphs.
\newblock {\em Geometriae Dedicata}, 144(1):115--128, (2010).

\bibitem{AFP}
J.~Aramayona, A.~Fossas, and H.~Parlier.
\newblock Arc and curve graphs for infinite-type surfaces.
\newblock {\em Proc. Amer. Math. Soc.}, 145(11):4995--5006, (2017).

\bibitem{AF}
J.~Aramayona and J.~F. Valdez.
\newblock On the {G}eometry of {G}raphs {A}ssociated to {I}nfinite-type
  surfaces.
\newblock https://arxiv.org/pdf/1605.05600.pdf, (2016).

\bibitem{BA}
J.~Bavard.
\newblock Hyperbolicit\'e du graphe des rayons et quasi-morphismes sur un gros
  groupe modulaire.
\newblock {\em Geom. Topol.}, 20(1):491--535, (2016).

\bibitem{BDR}
J.~Bavard, S.~Dowdall, and K.~Rafi.
\newblock Isomorphisms between big mapping class groups.
\newblock https://arxiv.org/pdf/1708.08383v2.pdf, (2017).

\bibitem{DF}
M.~G. Durham, F.~Fanoni, and N.~G. Vlamis.
\newblock Graphs of curves on infinite-type surfaces with mapping class group
  actions.
\newblock https://arxiv.org/pdf/1611.00841.pdf, (2017).

\bibitem{FM}
B.~Farb and D.~Margalit.
\newblock {\em A Primer on Mapping Class Groups}, volume~49 of {\em Princeton
  Mathematical Series}.
\newblock Princeton Univ. Press, (2012).

\bibitem{FP}
A.~Fossas and H.~Parlier.
\newblock Curve {G}raphs on {S}urfaces of {I}nfinite {T}ype.
\newblock {\em Ann. Acad. Sci. Fenn. Math.}, 40(2):793--801, (2015).

\bibitem{HMF}
J.~Hernández, I.~Morales, and J.~F. Valdez.
\newblock Isomorphisms between curve graphs of infinite-type surfaces are
  geometric.
\newblock https://arxiv.org/pdf/1706.03697.pdf, (2017).

\bibitem{HF}
J.~Hernández and J.~F. Valdez.
\newblock Automorphism groups of simplicial complexes of infinite-type
  surfaces.
\newblock {\em Publ. Mat.}, 61(1):51--82, (2017).

\bibitem{DM}
D.~Margalit.
\newblock Automorphisms of the pants complex.
\newblock {\em Duke Math. J.}, 121(3):457--479, (2004).

\bibitem{RM}
R.~Maungchang.
\newblock Finite rigid subgraphs of the pants graphs of punctured spheres.
\newblock https://arxiv.org/pdf/1303.3873.pdf, (2017).

\bibitem{MI}
Y.~Minsky.
\newblock A geometric approach to the complex of curves on a surface.
\newblock {\em Topology and Teichmüller spaces}, pages 149--158, (1996).

\bibitem{PM}
A.~O. Prishlyak and K.~I. Mischenko.
\newblock Classification of noncompact surfaces with boundary.
\newblock {\em Methods Funct. Anal. Topology}, 13(1):62--66, (2007).

\bibitem{R}
A.~Rasmussen.
\newblock Uniform hiperbolicity of the graphs of nonseparating curves via
  bicorn curves.
\newblock https://arxiv.org/pdf/1707.08283.pdf, (2017).

\bibitem{IR}
I.~Richards.
\newblock On the {C}lassification of {N}oncompact {S}urfaces.
\newblock {\em Trans. Amer. Math. Soc.}, 106:259--269, (1963).

\end{thebibliography}

\end{document}